\documentclass[letterpaper, 10 pt, conference]{ieeeconf}  

\IEEEoverridecommandlockouts                              

\usepackage{graphics} 
\usepackage{epsfig} 
\usepackage{times} 
\usepackage{amsmath} 
\usepackage{amssymb}  
\usepackage{dsfont}
\usepackage{booktabs}
\usepackage{mathtools}
\mathtoolsset{showonlyrefs}
\usepackage{soul} 
\usepackage{verbatim}

\usepackage{nicefrac}       
\usepackage{microtype}      
\usepackage{tabto}
\usepackage{xcolor}
\usepackage{cite}

\usepackage{enumerate}
\usepackage{caption}
\usepackage{subcaption}

\usepackage{relsize}

\newcommand{\Ldn}{L^\mathrm{d}}
\newcommand{\Lup}{L^\mathrm{u}}

\usepackage[notocbasic]{nomencl}
\makenomenclature
\usepackage{etoolbox}

\usepackage{algorithm}
\usepackage[noend]{algorithmic}

\usepackage{hyperref}       
\usepackage{url}            

\usepackage{xpatch}

\newtheorem{theorem}{Theorem}
\newtheorem{proposition}[theorem]{Proposition}

\newtheorem{lemma}[theorem]{Lemma}
\newtheorem{corollary}[theorem]{Corollary}

\title{\LARGE \bf
Optimizing Weighted Hodge Laplacian Flows on Simplicial Complexes
}

\author{Mathias Hudoba de Badyn and Tyler Summers
\thanks{MHdB is with DAS Lab and the Department of Technology Systems at the University of Oslo. TS is with the Department of Mechanical Engineering at University of Texas, Dallas. Emails: \texttt{mathias.hudoba@its.uio.no, tyler.summers@utdallas.edu}. This work was supported by the Research Council of Norway under the Centre for Space Systems and Sensors (309835) and the FME SOLAR (350244), and the United States Air Force Office of Scientific Research under Grants FA9550-23-1-0424 and FA2386-24-1-4014, and by the National Science Foundation under Grant ECCS-2047040.}
\\
}

\begin{document}

\maketitle
\thispagestyle{empty}
\pagestyle{empty}

\begin{abstract}
Simplicial complexes are generalizations of graphs that describe higher-order network interactions among nodes in the graph. Network dynamics described by graph Laplacian flows have been widely studied in network science and control theory, and these can be generalized to simplicial complexes using Hodge Laplacians. We study weighted Hodge Laplacian flows on simplicial complexes. In particular, we develop a framework for weighted consensus dynamics based on weighted Hodge Laplacian flows and show some decomposition results for weighted Hodge Laplacians. We then show that two key spectral functions of the weighted Hodge Laplacians, the trace of the pseudoinverse and the smallest non-zero eigenvalue, are jointly convex in upper and lower simplex weights and can be formulated as semidefinite programs. Thus, globally optimal weights can be efficiently determined to optimize flows in terms of these functions. Numerical experiments demonstrate that optimal weights can substantially improve these metrics compared to uniform weights.
\end{abstract}

\section{Introduction} \label{sec:intro} 

Dynamical systems distributed over networks have applications in robotics~\cite{Joordens2009}, natural science (eg., flocking of birds and schooling of fish)~\cite{vicsek1995}, fundamental physics~\cite{loll2019quantum}, opinion dynamics and social science~\cite{Altafini2012}, distributed optimization algorithms~\cite{wang2016fully}, and formation flight~\cite{Tanner2004}, among others.
The consensus protocol has had much success in allowing control architectures defined over these systems to come to agreement on common variables, such as a global coordinate system for formation flight, a common velocity vector for flocking, or agreement on a common dual variable in distributed optimization.
A natural question is how the structure of the graph affects underlying dynamics of the system on the network, and substantial attention to this question has been placed on the consensus protocol.

Of particular interest to the control community are performance characteristics, such as controllability/observability, and performance metrics, such as disturbance \& noise rejection such as the $\mathcal{H}_2$ and $\mathcal{H}_\infty$ norms.
A particularly elegant result is the characterization of symmetries in graphs stemming from graph automorphisms, which yield uncontrollable modes in consensus~\cite{Space2014}.
This was generalized to nonlinear variants of the consensus protocol~\cite{aguilar2014necessary,de2017controllability}.
Algorithms for generating controllable graphs were studied in~\cite{Hudobadebadyn2016}, and extensions of controllability analysis were performed on networks with positive and negative weights, identifying structural balance as the key ingredient on which symmetry imposes uncontrollability~\cite{Alemzadeh2017}.

System norms have physical interpretations on consensus networks.
The $\mathcal{H}_2$ norm can be interpreted as an effective resistance on graphs where the edge weights represent conductances on resistors (the reciprocal of resistance)~\cite{ghosh2008minimizing}.
This resistance interpretation was abused by~\cite{de2020mathcal} to define fast graph optimization algorithms on series-parallel graphs that use series-parallel formulae akin to adding resistances in series/parallel to compute and optimize the $\mathcal{H}_2$ norm.
The $\mathcal{H}_\infty$ norm forms a non-convex optimization problem on node and edge weights, but a tight relaxation was studied in~\cite{farhat2021h}.
Variations of these analyses on matrix-weighted networks (i.e., vector-valued consensus) were performed in~\cite{foight2020performance,de2020mathcal}.

Edge consensus is a variant of the consensus protocol which projects out the null space of the graph Laplacian of connected graphs by considering relative states between nodes, dubbed \emph{edge states}~\cite{zelazo2010edge}.
Optimization problems for edge weights and node weights (the latter also referred to as \emph{timescales}) were studied in~\cite{de2021graph}.
In another work, edge weights and timescales were shown to affect the $\mathcal{H}_2$ norm separately, leading to a \emph{separation principle} where edge weights and timescales can be chosen (or optimized) independently~\cite{foight2019time}.
In the present paper, we delve into a cohomological theory that shows this separation results from the seminal Hodge decomposition of simplicial complexes.

A simplicial complex is a generalization of a graph, and can be viewed as a hypergraph with additional structure imposed on it.
This `additional structure' allows a discrete calculus to be formally defined over simplicial complexes, and therefore symplectic (conservation law-preserving) discrete operators akin to higher-dimensional derivatives can be easily introduced.
This makes simplicial complexes and the topological theory surrounding them an important area of study in many computational domains, such as quantum gravity~\cite{loll2019quantum}, ranking in networks with higher-order interactions~\cite{hirani2010least}, and physics simulations~\cite{grady2010discrete}.

Consensus on simplicial complexes was first studied in~\cite{muhammad2006control}.
Recently, the field of topological signal processing has resulted in simplicial complexes being viewed as a natural data structure for data in which higher-order interactions (i.e., beyond pairwise) affect the dynamics of the system.
Seminal work in~\cite{barbarossa2020topological} studied applications of signals over simplicial complexes, and was extended to the study of finding weighted simplicial complexes from data in~\cite{battiloro2023topological}.
Cohomology and spectral analysis of weighted simplicial complexes and their associated Hodge Laplacian matrices was conducted by~\cite{wu2018weighted,horak2013spectra}.

The primary gap in the literature that this paper seeks to address is an understanding of how weights play a role in the performance of consensus on Laplacian flows over simplicial complexes, and if these weights can be optimized in a manner that is amenable to control-theoretic performance metrics.
The contributions of the present work are as follows:
\begin{enumerate}   \item A framework for weighted consensus dynamics based on weighted Hodge Laplacian flows (extending~\cite{muhammad2006control}, and the literature on consensus dynamics on graphs, eg.~\cite{Mesbahi2010})
    \item A statement and numerical calculation of the Hodge decomposition on weighted simplicial complexes
    \item Optimization problems and analysis thereof to find simplicial complex weights that optimize spectral functions of Hodge Laplacian flows.
\end{enumerate}

This paper is organized as follows.
Notation and preliminaries on simplicial complexes are in \S\ref{sobj:math-prel}.
Theoretical results on weighted simplicial complexes pertinent to studying Laplacian flows on simplicial complexes are presented in \S\ref{sobj:prop-weight-hodge}.
We discuss how to optimize weights on $k$-faces for various spectral functions related to generalizations of control-theoretic performance metrics in \S\ref{sobj:optim-hodge-lapl}.
Numerical results are presented in \S\ref{sobj:numerical-examples}, and the paper is concluded in \S\ref{sobj:concl-future-work}.

\section{Mathematical Preliminaries}
\label{sobj:math-prel}
We follow the notation and terminology of \cite{barbarossa2020topological}.
Note that the notation of~\cite{wu2018weighted,horak2013spectra} differs by an index and transpose in the definition of the boundary operator matrices.
\subsection{Simplicial Complexes}
 Given a set of $N$ points $\{v_0,\dots,v_{N-1}\}$, we define a \emph{$k$-simplex} $\sigma_i^k$ as a set of $k+1$ vertices $\{v_{i_0},\dots, v_{i_{k}}\}$ with $0\leq i_j \leq N-1$ for all $j=0,\dots, N-1$, and $v_{i_j}\neq v_{i_l}$ for all $i_j\neq i_l$.
The \emph{order} of a simplex is one less than its cardinality.
A \emph{face} of $\sigma_i^k$ is a set $\sigma_{i^k}\setminus \{v_{i_j}\}$ for some $v_{i_j}\in\sigma_i^k$, i.e. a $(k-1)$-simplex included in $\sigma_i^k$; each $\sigma_i^k$ has therefore exactly $k+1$ unique faces.
We can thus define a \emph{simplicial complex} $\mathcal{K}$ as a (finite) collection of simplices closed under inclusion: if $\sigma_i^k\in \mathcal{K}$, then each face of $\sigma_i^k\in \mathcal{K}$. 
The \emph{dimension} of a simplicial complex is the highest order over all simplicies in the complex.
The orientation of a face is captured by permutations of its elements -- two orientations are \emph{coherent} if an orientation can be constructed from another by an even number of permutations.
Given a face $\sigma_j^{k-1}\subset\sigma_j^k$, we write $\sigma_j^{k-1} \sim \sigma_i^k$ if the orientation of $\sigma_j^{k-1}$ is coherent with that of $\sigma_i^k$, and $\sigma_j^{k-1} \not\sim \sigma_i^k$ otherwise.

For each $k$, the vector space of (real) linear combinations of oriented $k$-faces of $\mathcal{K}$ is denoted $C_k(\mathcal{K},\mathbb{R})$, and the elements of $C_k(\mathcal{K},\mathbb{R})$ are referred to as \emph{$k$-chains}.
The \emph{boundary operator} is a mapping $\partial_k : C_k(\mathcal{K},\mathbb{R}) \mapsto C_{k-1}(\mathcal{K},\mathbb{R})$ used to define the Laplacian matrices for the Laplacian flows considered in this paper, and is given by
\begin{align}
  \partial_k \left[ v_{i_0},\dots,v_{i_k} \right] = \sum_{j=0}^k (-1)^j  \left[ v_{i_0},\dots,v_{i_k} \right]\setminus \left[ v_{i_j} \right].\label{eq:1}
\end{align}

\subsection{Algebraic Representations of Simplicial Complexes}

In the case of graphs, incidence matrices describe the edges between vertices, in particular which edge is assigned to which pair of indices.
The incidence matrices are then used to construct the graph Laplacian.
This can be generalized to the case of simplicial complexes as follows.

Let $\mathcal{K}$ be a simplicial complex of dimension $K$.
For each $k=1,\dots,K$, the incidence matrix $B_k$ will indicate which $k$-simplices are incident to which of its $k+1$, $(k-1)$-dimensional faces\footnote{
In the graph case, the columns of the incidence matrix indicate which two nodes (0-simplices) are incident to an edge (1-simplices).}.
We define the incidence matrix $B_k$ as,
\begin{align}
    B_k(i,j) = 
  \begin{cases}
    0 & \sigma_i^{k-1} \not\subset \sigma_j^k\\
    1 & \sigma_i^{k-1} \subset \sigma_j^k \text{ and } \sigma_i^{k-1}\sim \sigma_j^k\\
   -1 & \sigma_i^{k-1} \subset \sigma_j^k \text{ and } \sigma_i^{k-1}\not\sim \sigma_j^k\\
  \end{cases}
\end{align}
From \eqref{eq:1}, we can see that $B_kB_{k+1} = {\bf 0}$.
We now define the \emph{Hodge Laplacians} $L_k$ of order $k$ as follows:
\begin{align}
  L_0 &= B_1B_1^T\\
  L_k &= B_k^TB_k + B_{k+1}B_{k+1}^T,~k=1,\dots,K-1\\
  L_K &= B_K^TB_K.
\end{align}
The Hodge Laplacians of order $k=1,\dots,K-1$ contain two terms: the \emph{down Laplacian} $B_k^TB_k$, which encodes adjacency information about $k$-simplices to $(k-1)$-simplicies; and the up Laplacian $B_{k+1}B_{k+1}^T$, which encodes adjacency information about $k$ simplicies to $(k+1)$-simplicies\footnote{The up/down Laplacians are sometimes referred to as the upper/lower Laplacians in the literature}.

\subsection{Weighted Simplicial Complexes}
We first motivate the notion of \emph{weighted} simplicial complexes by introducing the less-general case on graphs.
In consensus on a graph $\mathcal{G}=(\mathcal{N},\mathcal{E})$, weights can be assigned to nodes $n\in \mathcal{N}$ and edges $e\in \mathcal{E}$ to form the \emph{weighted consensus dynamics},
\begin{align}
  \label{eq:2}
  w^n_i\dot x_i = \sum_{ij\in \mathcal{E}}w^e_{ij}(x_j - x_i).
\end{align}
The weights $w^n_i$ represent the \emph{timescale} of a node, and can be used to scale how fast the node integrates information from its neighbours~\cite{foight2019time,foight2020performance}, or to adjust for the relative timescale of the dynamical system under which node $i$ operates~\cite{farhat2021h}.
Node weights and edge weights can be optimized to improve the $\mathcal{H}_2$ and $\mathcal{H}_\infty$ performance metrics~\cite{de2020mathcal,de2021graph}.

In the present work, we extend consensus to \emph{weighted} simplicial complexes, and in doing so we follow the conventions set out in~\cite{wu2018weighted}.
To formally define a weighted simplicial complex, we define a weight function $w:\mathcal{K}\mapsto \mathbb{R}^+$ that assigns a weight $w(\sigma)$ to a face $\sigma \in \mathcal{K}$.
We define $W_i = \mathrm{diag}(w(\sigma_{i^1}),\dots,w(\sigma_{i^{n_i}}))$, where $n_i$ is the number of faces of order $i$ in $\mathcal{K}$.
Then, a natural inner product on simplicial complexes is the inner product on cochains $f,g\in C^i(\mathcal{K},\mathbb{R})$,
\begin{align}
  (f,g)_{C^i} = \sum_{i=1}^{n_i} w(\sigma_i)f(\sigma_i){g}(\sigma_i) = g^T W_i f.
\end{align}
The matrix representations of the boundary operators in the inner product space induced by $(\cdot,\cdot)_{C^i}$ will allow us to define Laplacians that capture a formal notion of consensus on weighted simplicial complexes.

The matrix representations of the weighted boundary operators $\partial_k, \partial_k^\dag$ are given in~\cite{wu2018weighted}.
\begin{theorem}[5.11 in~\cite{wu2018weighted}]
  Let $B_i$ denote the matrix representation of the unweighted boundary operator $\partial_i$, and let $W_i$ denote the (diagonal) matrix of face weights inducing the inner product on $C^i$.
Then, the matrix representation in the inner product space induced by $(\cdot,\cdot)_{C^i}$ of the \emph{weighted} boundary operator $\partial_i^\dag$ is $B_i^T$, and the matrix representation of $\partial_i$ is $W_{i-1}^{-1}B_iW_i$.
\end{theorem}

\begin{corollary}
\label{cor:lap}
Noting that $B_i = 0$ for $i<0$ and $i>k$ for a $k$-dimensional simplicial complex, the $k$th weighted up/down Laplacians are respectively,
  \begin{align}
    \Ldn_k &= B_{k}^TW_{k-1}^{-1}B_{k}W_k,~
             \Lup_k = W_k^{-1}B_{k+1}W_{k+1}B_{k+1}^T,
  \end{align}
  and the $k$th weighted Hodge Laplacian is therefore,
  \begin{align}
    L_k &= \Lup_k + \Ldn_k.
  \end{align} 
\end{corollary}
\begin{proof}
  We define the $\Lup_k,\Ldn_k$ as the matrix representations of $\partial_{k+1}\partial_{k+1}^\dag$ and $\partial_k^\dag\partial_k$ respectively.
\end{proof}

\subsubsection{Example of a Simplical Complex and Hodge Laplacians}
Consider an example of a simplex with faces ($k=2$).
The three Hodge Laplacians are,
\begin{align}
    L_0&= W_0^{-1}B_1W_1B_1^T \\
    L_1&= W_1^{-1}B_2W_2B_2^T + B_1^TW_0^{-1}B_1W_1\\
    L_2&= B^T_2W_1^{-1}B_2W_2.
\end{align}
Note that the dynamics on $x^0$,
\begin{align}
  \dot x^0 = -L_0x^0 = -W_0^{-1}B_1W_1B_1^T x^0
\end{align}
can be written as,
\begin{align}
  W_0\dot x^0 = -B_1W_1B_1^T x^0,
\end{align}
which are precisely the weighted consensus dynamics in Equation~\eqref{eq:2}.
Optimizing the weights $W_0$ and $W_1$ with respect to the $\mathcal{H}_2$ norm was discussed in~\cite{foight2019time}.
\subsubsection{Example of a Simplicial Complex and Edge Consensus}
Consider an example of a simplex with edges ($k=1$).
The two Hodge Laplacians are,
\begin{align}
    L_0&= W_0^{-1}B_1W_1B_1^T \\
    L_1&=  B_1^TW_0^{-1}B_1W_1.
\end{align}
The Laplacian $L_1$ is called the \emph{edge Laplacian}, and edge consensus has been studied extensively in\cite{zelazo2010edge}.
Optimizing weights $W_0$ and $W_1$ with respect to the $\mathcal{H}_2$ norm in this setting was analyzed in~\cite{de2021graph}.

\section{Properties of Weighted Hodge Laplacians}
\label{sobj:prop-weight-hodge}
In this section, we discuss the spectral properties of weighted Hodge Laplacians, beginning with properties of the eigenvalues and eigenvectors thereof.
We then discuss the numerical computation of the Hodge decomposition, which allows the decomposition of $k$-Laplacian flows into dynamics on the $(k-1)$-simplices and $(k+1)$-simplices, and the nullspace of $L_k$.

\begin{lemma}[Prop.~2.6 in \cite{wu2018weighted}]
  Let $\partial_k$ denote the weighted $k$-boundary operator. Then,
  $\partial_k\partial_{k+1} = 0$.
\end{lemma}

\begin{theorem}[Weighted analogue of Prop.~1 in~\cite{barbarossa2020topological}]
  \label{thr:1}
The Laplacian matrices $L_k$ and the associated up/down Laplacians $\Lup_k,\Ldn_k$ satisfy:
\begin{enumerate}
  \item the eigenvectors of the nonzero eigenvalues of $\Ldn_k$ are orthogonal to the eigenvectors of nonzero eigenvalues of $\Lup_k$, and vice versa.
  \item If $v$ is an eigenvector of $\Lup_{k-1}$, then $B^T_kv$ is an eigenvector of $\Ldn_k$ with the same eigenvalue.
  \item The eigenvectors associated with the nonzero eigenvalues $\lambda$ of $L_k$ are either the eigenvectors of $\Ldn_k$ or $\Lup_k$.
  \item The nonzero eigenvalues of $L_k$ are either the eigenvalues of $\Ldn_k$ or $\Lup_k$.
\end{enumerate}
\end{theorem}
\begin{proof}
  Property 1 follows from a simple calculation. Suppose $\Ldn_k v = \lambda v$.
Then,
\begin{align}
  \Lup_k\lambda v &= \Lup_k\Ldn_k v\\
 &= W_k^{-1}B_{k+1}W_{k+1}\underbrace{B_{k+1}^TB_k^T}_{=0} W_{k+1}^{-1} B_kW_kv=0.
\end{align}
For Property 2, suppose that $\Lup_{k-1} v = \lambda v$.
Then,
\begin{align}
  \Ldn_k B^T_k v = B_k^TW_{k-1}^{-1}B_kW_kB_k^Tv = B_k^T\Lup_{k-1}v = \lambda B_k^T\lambda v.
\end{align}

Properties 3 \& 4 follow directly from Property 1.
\end{proof}

\begin{theorem}{(Hodge Decomposition of Weighted Simplicial Complexes)}
    \label{thm:hodge_decomp}
    Let $s^k \in C_k(\mathcal{K},\mathbb{R})$ be a $k$-chain on a simplicial complex $\mathcal{K}$.
    The vector space $\mathbb{R}^{D_k}$, where $D_k$ is the number of simplices of order $k$ in $\mathcal{K}$, can be decomposed as,
    \begin{align}
      \mathbb{R}^{D_k} = \mathrm{Im}(\partial_k^\dag)\oplus\mathrm{Ker}(L_k)\oplus\mathrm{Im}(\partial_{k+1})
    \end{align}
    and therefore a signal $s^k$ defined on the $k$-simplices of $\mathcal{K}$ can be decomposed as,
    \begin{align}
      s^k = B_k^Ts^{k-1} +s^k_H + W_k^{-1}B_{k+1}W_{k+1}s^{k+1}.
    \end{align}
\end{theorem}
We can compute the Hodge decomposition by extending the result in \cite{hirani2010least} to the weighted Laplacian case.
\begin{theorem}[Weighted analog of Thm 6.2 in~\cite{hirani2010least}]
    \label{thm:num_hodge_decomp}
    The following are equivalent:
    \begin{enumerate}
      \item There exists a decomposition of a signal $s^k\in C_i(\mathcal{K},\mathbb{R})$ on $k$ faces of the form,
        \begin{align}
           s^k &= B_k^Ts^{k-1} +s^k_H + W_k^{-1}B_{k+1}W_{k+1}s^{k+1}\\
           s^k_H &\in \mathrm{Ker}(L_k)
        \end{align}
        \item There exists a solution $\{\alpha,\beta\} = \{s^{k-1},s^{k+1}\}$ to the equations,
          \begin{align}
            \Lup_{k}\alpha &= W_{k-1}^{-1}B_kW_k s^k \label{eq:lap1}\\
            \Ldn_{k+1}\beta &= B_{k+1}^T s^k \label{eq:lap2}
          \end{align}
        \item The following pair of least squares problems has the solution $\{\alpha,\beta\} = \{s^{k-1},s^{k+1}\}$, where $\|s^k\|_{W_k}^2 = (s^k,s^k)_{W_k}$ is the norm induced by the $k$-cochain inner product $(\cdot,\cdot)_{W_k}$: 
          \begin{align}
            &\begin{array}{ll}
               \min_\alpha & \| \partial_1\alpha - s^k\|_{W_k}^2
             \end{array} \label{eq:ls1} \\ 
            &\begin{array}{ll} 
               \min_\beta & \| \partial_2^T\beta - s^k\|_{W_k}^2 
             \end{array} \label{eq:ls2}
          \end{align}
    \end{enumerate}
    
\end{theorem}
\begin{proof}
  The optimality conditions of Equations~\eqref{eq:ls1} and~\eqref{eq:ls2} are respectively, 
\begin{align}
  \partial_k\partial_k^T \alpha &= \partial_k s^k\\
  \partial_{k+1}^T\partial_{k+1} \beta &= \partial_{k+1}^T s^k,
\end{align}
and by substituting the identities in Corollary~\ref{cor:lap}, we have that
\begin{align}
  \Lup_k \alpha &= W_{k-1}^{-1}B_kW_k s^k\label{eq:int1}\\
  \Ldn_{k+1} \beta &= B_{k+1}^T s^k.\label{eq:int2}
\end{align}
It follows that $B_k^T\alpha \in \mathrm{Im}(\partial_k^\dag)$ and $W_k^{-1}B_{k+1}W_{k+1}\beta \in \mathrm{Im}(\partial_{k+1})$.
By Theorem~\ref{thr:1}.1, it follows that $s^K_H\in \mathrm{ker}(L_k)$ is orthogonal to both $B_k^Ts^{k-1}$ and $W_k^{-1}B_{k+1}W_{k+1}s^{k+1}$, completing the proof.


\end{proof}

\section{Optimizing Hodge Laplacian Flows for Simplicial Complexes}
\label{sobj:optim-hodge-lapl}
In this section, we consider the problem of optimizing Hodge Laplacian flows for simplicial complexes. We show that two key spectral functions of the Hodge Laplacian, the trace of the pseudoinverse and the smallest non-zero eigenvalue, are convex in the upper and lower weights and can be formulated as semidefinite programs. Thus, globally optimal weights can be efficiently determined to optimize flows in terms of these functions. These results generalize some of the results in \cite{ghosh2008minimizing} to simplicial complexes.



\subsection{Hodge Laplacian Flows}
Hodge Laplacian flows represent signals that evolve on simplicial complexes according to the differential equation
\begin{equation} \label{Hodgeflows}
    \dot{x}^k(t) = -L_k x^k(t), 
\end{equation}
where $x^k(t)\in C_k(\mathcal{K},\mathbb{R})\times \mathcal{T} $ is a time-dependent state associated with $k$-simplices (here, $\mathcal{T}$ denotes a time domain). For example, $L_0$ governs vertex flows, $L_1$ governs edge flow, $L_2$ governs face flows, and so on. The spectrum of the Hodge Laplacians provides important information about the flows \eqref{Hodgeflows}. The eigenvectors of $L_k$ reveal mode shapes, or flow patterns, while the eigenvalues specify decay rates for each mode. The kernel of $L_k$ shows fundamental structure of the flows, representing harmonic $k$-forms that do not decay over time. 
Consensus dynamics on simplicial complexes governed by Hodge Laplacian flows on \emph{unweighted} simplicial complexes were first studied in\cite{muhammad2006control}.

From Theorem~\ref{thm:hodge_decomp}, we have the following characterization of flows on weighted Hodge decompositions.
\begin{corollary}[Hodge Laplacian Flow Decomposition]
  Let $x^k(t)\in C_k(\mathcal{K},\mathbb{R})\times \mathcal{T} $ be a Laplacian flow according to Equation~\eqref{Hodgeflows}.
  Then, $x^k(t)$ evolves according to the differential equation,
  \begin{align}
    \dot{x}^k(t) = \Ldn_k x^{k-1}(t) + \Lup_{k+1}x^{k+1}(t),
  \end{align}
  where $x^{k-1}(t) \in C_{k-1}(\mathcal{K}, \mathbb{R})\times \mathcal{T}$, $x^{k+1}(t) \in C_{k+1}(\mathcal{K}, \mathbb{R})\times \mathcal{T}$ are flows on $(k-1)$-faces and $(k+1)$-faces, respectively.
\end{corollary}

We consider two key spectral functions of the Hodge Laplacian that govern flow dynamics: the trace of the pseudoinverse and the smallest non-zero eigenvalue. The trace of the pseudoinverse of the Hodge Laplacian is given by
\begin{equation}
    \mathbf{tr} L_k^+ = \sum_{\lambda_i \neq 0} \frac{1}{\lambda_i}.
\end{equation}
This function is related to the $\mathcal{H}_2$ norm of the flow, which can be interpreted as a measure of the spread of the flow relative to harmonic flows in the kernel. For the graph Laplacian $L_0$, it is proportional to the effective resistance of the graph. The smallest non-zero eigenvalue $\lambda_{\text{min}}(L_k) = \min_{\lambda \neq 0} \lambda(L_k)$ determines how fast the slowest non-harmonic flow dissipates.

\subsection{Optimizing Hodge Laplacian Flows}
We show that the trace of the Hodge Laplacian $L_k$ pseudoinverse and the smallest non-zero eigenvalue are convex functions of the upper and lower weights and can be formulated as a semidefinite programs. We start with the following result.
\begin{proposition}
Consider the weighted Hodge Laplacian $L_k = B_k^T W_{k-1}^{-1} B_k W_k + W_k^{-1}B_{k+1} W_{k+1} B_{k+1}^T$, where the $k$-level weights $W_k$ are fixed and the lower and upper weights $W_{k-1} = \text{diag} (w_{k-1})$ and $W_{k+1} = \text{diag} (w_{k+1})$ are variables. Let $K$ be a matrix whose columns form an orthonormal basis for the kernel of $L_k$. The problem of optimizing the weights to minimize the Hodge Laplacian pseudoinverse can be expressed as the semidefinite program
\begin{equation} \label{sdp-traceinv}
    \begin{aligned}
        &\text{minimize} \quad &&\mathbf{tr} \ Y \\
        &\text{subject to} \quad && \mathbf{1}^T w_{k-1} = 1 , \quad \mathbf{1}^T w_{k+1} = 1 \\
        & && w_{k-1} \geq 0, \quad w_{k+1} \geq 0 \\
        & && \left[\begin{array}{cc} L_k + K K^T & I \\I & Y\end{array}\right] \succeq 0,
    \end{aligned}
\end{equation}
where the variables are the weights $w_{k-1}$ and $w_{k+1}$ and the slack symmetric matrix $Y$.
\end{proposition}
\begin{proof}
The pseudoinverse of the Hodge Laplacian can be expressed as
\begin{equation}
    L_k^+ = (L_k + K K^T)^{-1} - K K^T.
\end{equation}
Thus, $\mathbf{tr}(L_k^+) = \mathbf{tr}((L_k + K K^T)^{-1}) - m$, where $m$ is the dimension of the kernel of $L_k$. (If $L_k$ is full rank, then $m=0$ and $K$ can be omitted. Since $L_k + K K^T \succ 0$, then by the Schur complement
\begin{equation} \label{blockschur}
    \left[\begin{array}{cc} L_k + K K^T & I \\I & Y\end{array}\right] \succeq 0 \quad \Leftrightarrow \quad Y \succeq (L_k + K K^T)^{-1}.
\end{equation}
The optimal value for the slack variable is $Y = L_k + K K^T)^{-1}$, so the objective becomes $\mathbf{tr} L_k^+ + m$. Since Hodge Laplacian is linear in the upper and lower weights $w_{k+1}$ and $w_{k-1}$ when the $k$-level weights $w_k$ are fixed, the left side of \eqref{blockschur} is a linear matrix inequality. Thus, the problem \eqref{sdp-traceinv} is a semidefinite program.
\end{proof}

We have the following result for the smallest nonzero eigenvalue.
\begin{proposition}
Consider the weighted Hodge Laplacian $L_k = B_k^T W_{k-1}^{-1} B_k W_k + W_k^{-1}B_{k+1} W_{k+1} B_{k+1}^T$, where the $k$-level weights $W_k$ are fixed and the lower and upper weights $W_{k-1} = \text{diag} (w_{k-1})$ and $W_{k+1} = \text{diag} (w_{k+1})$ are variables. Let $K$ be a matrix whose columns form an orthonormal basis for the kernel of $L_k$. The problem determining weights to maximize the smallest non-zero eigenvalue of the Hodge Laplacian  can be formulated as the semidefinite program
\begin{equation} \label{sdp-mineig}
    \begin{aligned}
        &\text{maximize} \quad &&\gamma \\
        &\text{subject to} \quad && \mathbf{1}^T w_{k-1} = 1 , \quad \mathbf{1}^T w_{k+1} = 1 \\
        & && w_{k-1} \geq 0, \quad w_{k+1} \geq 0 \\
        & && \gamma I \preceq L_k + \beta K K^T.
    \end{aligned}
\end{equation}
The variables are the weights $w_{k-1}$ and $w_{k+1}$ and the slack variables $\gamma$ and $\beta$. 
\end{proposition}
\begin{proof}
    For any symmetric matrix $A$, the constraint $\gamma I \preceq A$ ensures that all eigenvalues of $A$ (including the smallest one) are at least $\gamma$. Since the Hodge Laplacian $L_k$ may be rank deficient due to structural ``holes'' in the associated simplicial complex, and we would like our optimization problem to focus on the smallest non-zero eigenvalue, we include the constraint $\gamma I \preceq L_k + \beta K K^T$. The term $\beta K K^T$ shifts the zero eigenvalues of $L_k$ away from the origin, which are independent of the weights. The constraint ensures that the unconstrained slack variable $\beta$ satisfies $\beta \geq \gamma$ so that at optimality, the minimum non-zero eigenvalue of $L_k$ is $\gamma$.
\end{proof}

\section{Numerical Examples}
\label{sobj:numerical-examples}
In this section, we illustrate the computation of optimal weights for simplicial complexes based on the results in Section III. To obtain problem instances, we generate Vietoris Rips complexes. A Vietoris Rips complex is a type of simplicial complex built from a set of points based on pairwise distances and can be used to explore the higher-order topological and algebraic structure of networks and datasets. When the points are sampled randomly from a metric space, Vietoris Rips complexes can be viewed as a generalization of gemometric random graphs.

Let $p = [p_1, p_2, ..., p_N]$ be a set of $N$ points in a metric space. For a given threshold $\epsilon$, a Vietoris Rips complex is constructed as follows. Each point $p_i$ represents a vertex. A pair of points forms an edge if $d(p_i, p_j) \leq \epsilon$. Each $k$-simplex up to a desired order included in the complex if all of its vertices are pairwise connected by edges. Thus, if a set of $k+1$ points forms a fully connected subgraph in the associated distance graph, it is included as a $k$-simplex.

To illustrate the computation of optimal higher-order weights described in Section III, we generated random Vietoris Rips complexes by sampling points uniformly on the unit square. Fig. 1 shows a Vietoris Rips complex based on 30 points with a threshold $\epsilon = 0.5$, which has 30 vertices, 151 edges, and 339 faces. We solved the semidefinite program \eqref{sdp-traceinv} for the weighted Hodge Laplacian $L_1$ to compute face weights to minimize the pseudoinverse with the vertex and edge weights all fixed to unity. Fig. 2 shows the optimal face weights for the 339 faces, which improve the pseudoinverse by 8.9\% compared to the unweighted Hodge Laplacian (with unity face weights). We also solved the semidefinite program \eqref{sdp-mineig} to maximize the smallest non-zero eigenvalue of the weighted Hodge Laplacian $L_1$. Fig. 3 shows the optimal weights with the same face ordering as in Fig. 2. These weights increased the smallest non-zero eigenvalue by 228\% compared to the unweighted Hodge Laplacian. It can be seen that the weights are substantially different from the ones that optimize the pseudoinverse trace. Fig. 4 shows the corresponding edge flows (solutions of $\dot x^1(t) = -L_1 x^1(t)$) for the unweighted and optimally weighted cases. It can be seen that the edge flow with optimal weights exhibits much faster convergence.
\begin{figure}[tb]
    \centering
    \includegraphics[width=\linewidth]{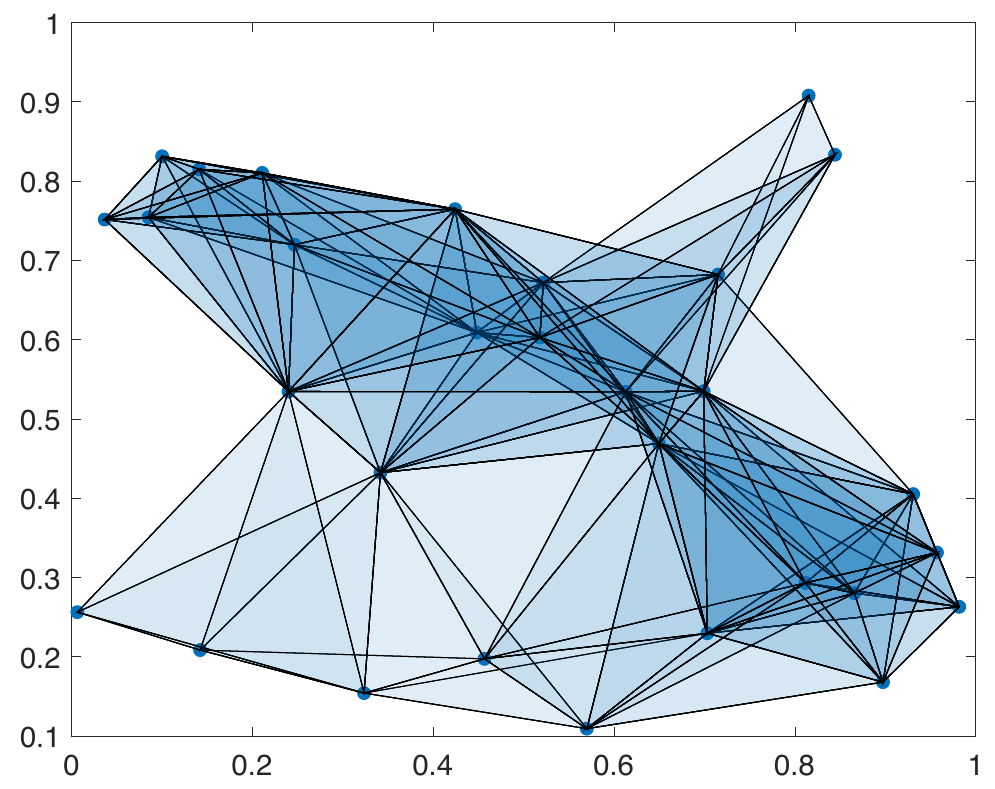}
    \caption{A Vietoris Rips complex generated from 30 points sampled uniformly on the unit square with threshold 0.5. The complex contains 30 nodes, 151 edges, and 339 faces. Each face in the complex is shaded according to the optimal face weights computed with the semidefinite program \eqref{sdp-traceinv} for the weighted Hodge Laplacian $L_1$.}
    \label{fig:VR}
\end{figure}
\begin{figure}[!]
    \centering
    \includegraphics[width=\linewidth]{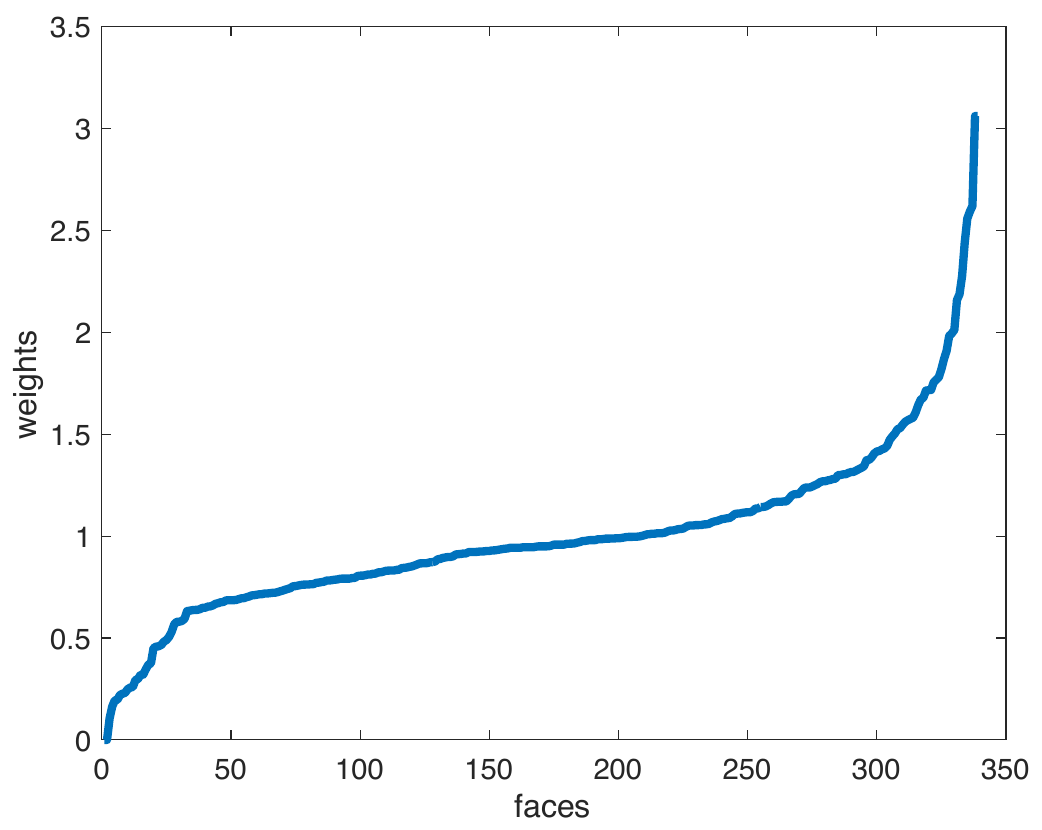}
    \caption{Optimal face weights for the 339 faces of the Vietoris Rips complex from Fig. 1. The optimal weights provide a 8.9\% improvement of the Hodge Laplacian pseudoinverse compared to the unweighted Hodge Laplacian (with unity weights).}
    \label{fig:VR}
\end{figure}
\begin{figure}[tb]
    \centering
    \includegraphics[width=\linewidth]{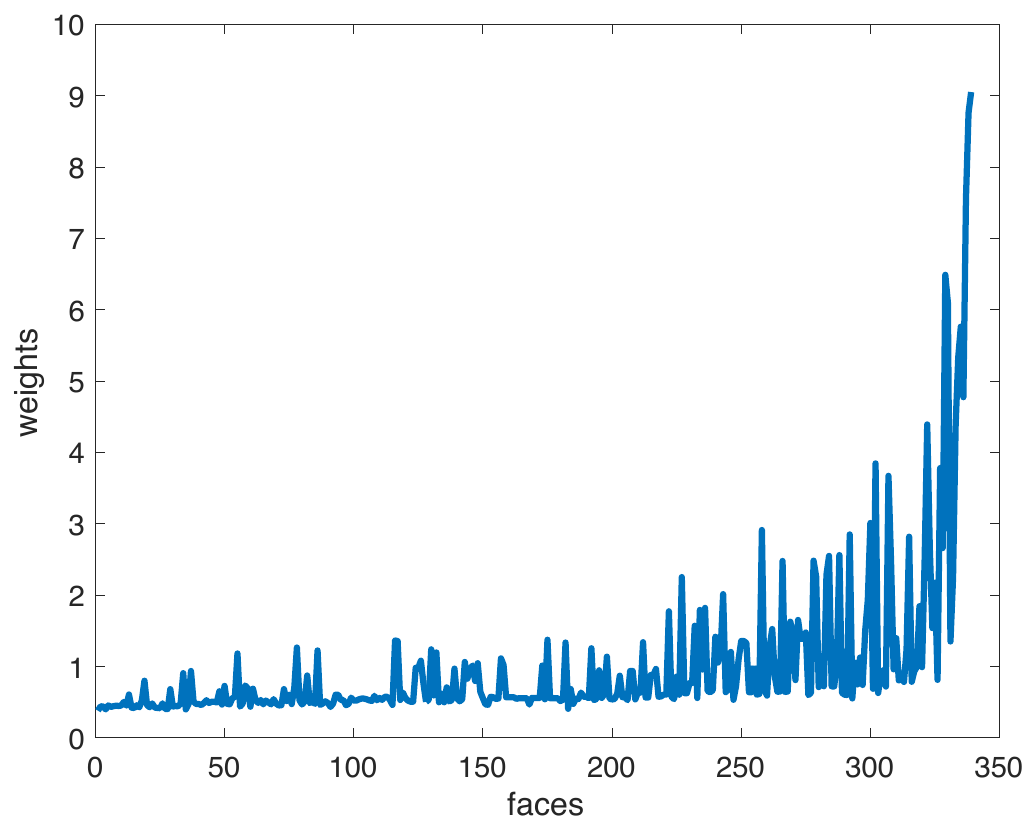}
    \caption{Optimal face weights for the 339 faces of the Vietoris Rips complex from Fig. 1 computed with the semidefinite program \eqref{sdp-mineig} to maximize the smallest non-zero eigenvalue of the weighted Hodge Laplacian $L_1$, with the same face ordering as in Fig. 2. The optimal weights provide a 228\% increase of the smallest non-zero eigenvalue compared to the unweighted Hodge Laplacian (with unity weights).}
    \label{fig:VR}
\end{figure}
\begin{figure}[tb]
    \centering
    \includegraphics[width=\linewidth]{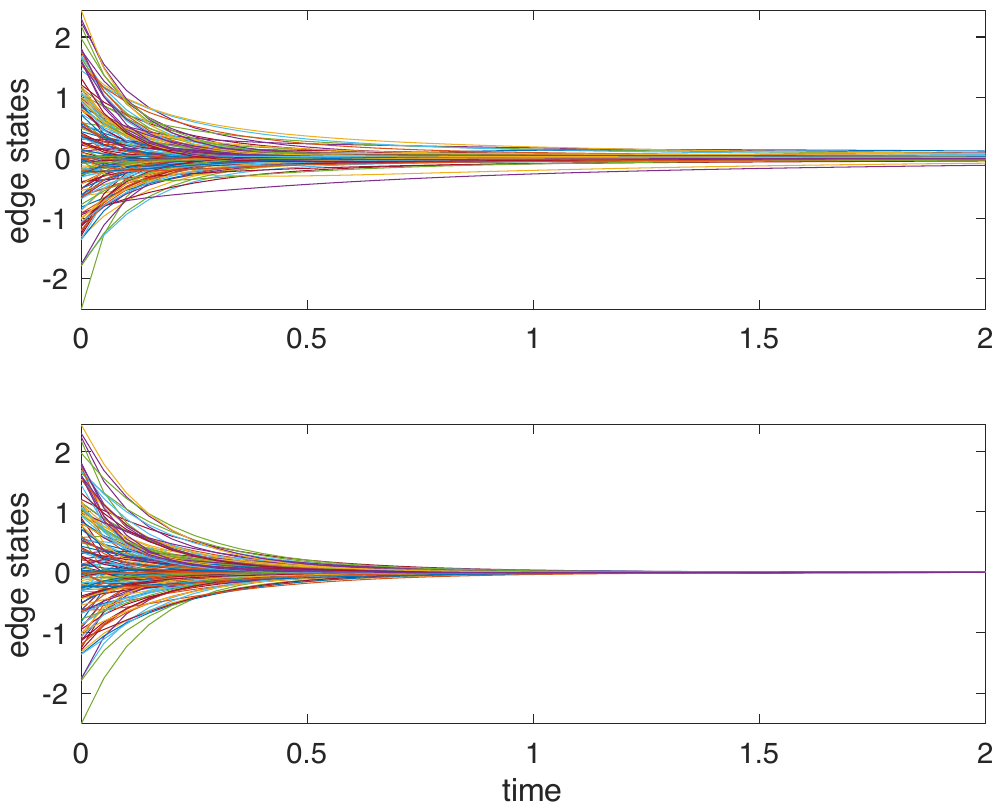}
    \caption{(top) Edge flows for the unweighted Hodge Laplacian; (bottom) edge flows for the weighted Hodge Laplacian with optimal weights computed with the semidefinite program \eqref{sdp-traceinv} to maximize the smallest non-zero eigenvalue of the weighted Hodge Laplacian $L_1$, exhibiting much faster convergence.}
    \label{fig:VR}
\end{figure}

\section{Conclusions and Future Work}
\label{sobj:concl-future-work}
In this paper, we developed a framework for weighted consensus dynamics based on weighted Hodge Laplacian flows and demonstrated some decomposition results. We showed that two key spectral functions of the weighted Hodge Laplacians, the trace of the pseudoinverse and the smallest non-zero eigenvalue, are jointly convex in upper and lower simplex weights and can be formulated as semidefinite programs and thus can be globally optimized. Future work will explore further system theoretic properties of dynamics on simplicial complexes and develop additional complementary optimization algorithms for designing simplicial complexes with desired properties.
\bibliographystyle{IEEEtran}
\bibliography{bibliography.bib}

@book{grady2010discrete,
  title={Discrete calculus: Applied analysis on graphs for computational science},
  author={Grady, Leo J and Polimeni, Jonathan R},
  volume={3},
  year={2010},
  publisher={Springer}
}

@article{horak2013spectra,
  title={Spectra of combinatorial {L}aplace operators on simplicial complexes},
  author={Horak, Danijela and Jost, J{\"u}rgen},
  journal={Advances in Mathematics},
  volume={244},
  pages={303--336},
  year={2013},
  publisher={Elsevier}
}

@article{loll2019quantum,
  title={Quantum gravity from causal dynamical triangulations: a review},
  author={Loll, Renate},
  journal={Classical and Quantum Gravity},
  volume={37},
  number={1},
  pages={013002},
  year={2019},
  publisher={IOP Publishing}
}

@article{vicsek1995,
    title = {{Novel type of phase transistion in a system of self-driven particles}},
    year = {1995},
    journal = {New York},
    author = {Vicsek, Tamas and Czirok, Andras and Ben-Jacob, Eshel and Cohen, Inon and Shochet, Ofer},
    number = {6},
    pages = {1226--1229},
    volume = {75},
    isbn = {0031-9007{\textbackslash}r1079-7114},
    doi = {10.1103/PhysRevLett.74.2248},
    issn = {1079-7114},
    pmid = {10059141},
    arxivId = {arXiv:1011.1669v3}
}

@inproceedings{de2017controllability,
  title={Controllability and data-driven identification of bipartite consensus on nonlinear signed networks},
  author={Hudoba de Badyn, Mathias and Alemzadeh, Siavash and Mesbahi, Mehran},
  booktitle={2017 IEEE 56th Annual Conference on Decision and Control (CDC)},
  pages={3557--3562},
  year={2017},
  organization={IEEE}
}

@inproceedings{aguilar2014necessary,
  title={Necessary conditions for controllability of nonlinear networked control systems},
  author={Aguilar, Cesar O and Gharesifard, Bahman},
  booktitle={2014 american control conference},
  pages={5379--5383},
  year={2014},
  organization={IEEE}
}

@article{Tanner2004,
    title = {{Leader-to-formation stability}},
    year = {2004},
    journal = {IEEE Transactions on Robotics and Automation},
    author = {Tanner, H G and Pappas, G J and Kumar, V},
    number = {3},
    pages = {443--455},
    volume = {20},
    doi = {10.1109/TRA.2004.825275},
    issn = {1042-296X}
}

@article{wang2016fully,
  title={A fully-decentralized consensus-based ADMM approach for DC-OPF with demand response},
  author={Wang, Yamin and Wu, Lei and Wang, Shouxiang},
  journal={IEEE Transactions on Smart Grid},
  volume={8},
  number={6},
  pages={2637--2647},
  year={2016},
  publisher={IEEE}
}

@article{Altafini2012,
    title = {{Dynamics of opinion forming in structurally balanced social networks}},
    year = {2012},
    journal = {PloS one},
    author = {Altafini, Claudio},
    number = {6},
    pages = {5876--5881},
    volume = {7}
}

@inproceedings{Joordens2009,
    title = {{Underwater swarm robotics consensus control}},
    year = {2009},
    booktitle = {Proc. IEEE International Conference on Systems, Man and Cybernetics},
    author = {Joordens, Matthew A. and Jamshidi, Mo},
    number = {October},
    pages = {3163--3168},
    address = {San Antonio},
    isbn = {9781424427949},
    doi = {10.1109/ICSMC.2009.5346165},
    issn = {1062922X}
}

@inproceedings{foight2019time,
  title={Time scale design for network resilience},
  author={Foight, Dillon R and Hudoba de Badyn, Mathias and Mesbahi, Mehran},
  booktitle={2019 IEEE 58th Conference on Decision and Control (CDC)},
  pages={2096--2101},
  year={2019},
  organization={IEEE}
}

@article{farhat2021h,
  title={H$_\infty$ network optimization for edge consensus},
  author={Farhat, Omar and Abou Jaoude, Dany and Hudoba de Badyn, Mathias},
  journal={European Journal of Control},
  volume={62},
  pages={2--13},
  year={2021},
  publisher={Elsevier}
}

@article{hirani2010least,
  title={Least squares ranking on graphs},
  author={Hirani, Anil N and Kalyanaraman, Kaushik and Watts, Seth},
  journal={arXiv preprint arXiv:1011.1716},
  year={2010}
}

@book{Mesbahi2010,
    title = {{Graph-Theoretic Methods in Multiagent Networks}},
    year = {2010},
    booktitle = {Princeton Univ Pr},
    author = {Mesbahi, Mehran and Egerstedt, Magnus},
    pages = {403},
    publisher = {Princeton University Press},
    isbn = {9780691140612},
    doi = {10.1073/pnas.0703993104},
    issn = {1098-6596},
    pmid = {25246403},
    arxivId = {arXiv:1011.1669v3}
}

@inproceedings{battiloro2023topological,
  title={Topological signal processing over weighted simplicial complexes},
  author={Battiloro, Claudio and Sardellitti, Stefania and Barbarossa, Sergio and Di Lorenzo, Paolo},
  booktitle={ICASSP 2023-2023 IEEE International Conference on Acoustics, Speech and Signal Processing (ICASSP)},
  pages={1--5},
  year={2023},
  organization={IEEE}
}

@article{wu2018weighted,
  title={Weighted (co) homology and weighted {L}aplacian},
  author={Wu, Chengyuan and Ren, Shiquan and Wu, Jie and Xia, Kelin},
  journal={arXiv preprint arXiv:1804.06990},
  year={2018}
}

@article{barbarossa2020topological,
  title={Topological signal processing over simplicial complexes},
  author={Barbarossa, Sergio and Sardellitti, Stefania},
  journal={IEEE Transactions on Signal Processing},
  volume={68},
  pages={2992--3007},
  year={2020},
  publisher={IEEE}
}

@article{foight2020performance,
  title={Performance and design of consensus on matrix-weighted and time-scaled graphs},
  author={Foight, Dillon R and Hudoba de Badyn, Mathias and Mesbahi, Mehran},
  journal={IEEE Transactions on Control of Network Systems},
  volume={7},
  number={4},
  pages={1812--1822},
  year={2020},
  publisher={IEEE}
}

@article{de2020mathcal,
  title={$\mathcal{H}_2$ Performance of Series-Parallel Networks: A Compositional Perspective},
  author={Hudoba de Badyn, Mathias and Mesbahi, Mehran},
  journal={IEEE Transactions on Automatic Control},
  volume={66},
  number={1},
  pages={354--361},
  year={2020},
  publisher={IEEE}
}

@inproceedings{muhammad2006control,
  title={Control using higher order {L}aplacians in network topologies},
  author={Muhammad, Abubakr and Egerstedt, Magnus},
  booktitle={Proc. of 17th International Symposium on Mathematical Theory of Networks and Systems},
  pages={1024--1038},
  year={2006},
  organization={Citeseer}
}

@article{ghosh2008minimizing,
  title={Minimizing effective resistance of a graph},
  author={Ghosh, Arpita and Boyd, Stephen and Saberi, Amin},
  journal={SIAM review},
  volume={50},
  number={1},
  pages={37--66},
  year={2008},
  publisher={SIAM}
}

@article{de2021graph,
  title={Graph-theoretic optimization for edge consensus},
  author={Hudoba de Badyn, Mathias and Foight, Dillon R and Calderone, Daniel and Mesbahi, Mehran and Smith, Roy S},
  journal={IFAC-PapersOnLine},
  volume={54},
  number={9},
  pages={533--538},
  year={2021},
  publisher={Elsevier}
}

@article{zelazo2010edge,
  title={Edge agreement: Graph-theoretic performance bounds and passivity analysis},
  author={Zelazo, Daniel and Mesbahi, Mehran},
  journal={IEEE Transactions on Automatic Control},
  volume={56},
  number={3},
  pages={544--555},
  year={2010},
  publisher={IEEE}
}

@article{Space2014,
    title = {{Controllability of Multi-Agent Systems From a Graph-Theoretic Perspective}},
    year = {2014},
    author = {Rahmani, Amirreza},
    number = {6},
    pages = {2599--2622},
    volume = {52},
    isbn = {2010001753}
}

@inproceedings{Alemzadeh2017,
    title = {{Controllability and stabilizability analysis of signed consensus networks}},
    year = {2017},
    booktitle = {Proc. IEEE Conference on Control Technology and Applications},
    author = {Alemzadeh, Siavash and Hudoba de Badyn, Mathias and Mesbahi, Mehran},
    pages = {55--60},
    address = {Kohala Coast, USA}
}

@inproceedings{Hudobadebadyn2016,
    title = {{Growing controllable networks via whiskering and submodular optimization}},
    year = {2016},
    booktitle = {Proc. 55th IEEE Conference on Decision and Control},
    author = {Hudoba de Badyn, Mathias and Mesbahi, Mehran},
    pages = {867--872},
    address = {Las Vegas, USA},
    isbn = {9781509018369},
    doi = {10.1109/CDC.2016.7798376},
    arxivId = {1609.08733}
}

\end{document}